\documentclass[12pt]{amsart}

\usepackage[english]{babel}
\usepackage{amssymb} 
\usepackage{amsfonts} 
\usepackage{amsmath}
\usepackage{amsthm} 
\usepackage{epsfig} 
\usepackage{color}
\usepackage{pinlabel}
\usepackage{amscd}
\usepackage[all]{xy}
\usepackage{hyperref}
\usepackage{subcaption}
\usepackage{graphicx}

\newtheorem{lemma}{Lemma}[section]
\newtheorem{theorem}[lemma]{Theorem}
\newtheorem{prop}[lemma]{Proposition}
\newtheorem{cor}[lemma]{Corollary}

\theoremstyle{definition}
\newtheorem{defn}[lemma]{Definition}

\theoremstyle{remark}
\newtheorem{rem}[lemma]{Remark} 

\newcommand{\Iso}{{\mathrm{Isom}}}

\newcommand{\cro} {\ensuremath {{\mathrm{cr}}}}

\newcommand{\matR} {\ensuremath {\mathbb{R}}}

\newcommand{\matC} {\ensuremath {\mathbb{C}}}

\newcommand{\calT} {\ensuremath {\mathcal{T}}}

\newcommand\evol{\operatorname{\varepsilon-vol}}
\newcommand\bvol{\operatorname{\partial-vol}}

\usepackage[letterpaper,left=1in,right=1in,top=1in,bottom=0.9in]{geometry}
\let\oldbibliography\thebibliography
\renewcommand{\thebibliography}[1]{%
  \oldbibliography{#1}%
  \setlength{\itemsep}{.4\baselineskip}%
}

\author{Alexander Kolpakov}
\address{Institut de math\'ematiques, Rue Emile-Argand 11, 2000 Neuch\^atel, Switzerland}
\email{kolpakov (dot) alexander (at) gmail (dot) com}

\author{Alan W. Reid}
\address{Department of Mathematics,  Rice University, 6100 Main Street, Houston, TX 77005, USA}
\email{alan (dot) reid (at) rice (dot) edu}

\author{Stefano Riolo}
\address{Institut de math\'ematiques, Rue Emile-Argand 11, 2000 Neuch\^atel, Switzerland}
\email{stefano (dot) riolo (at) unine (dot) ch}

\title[Many cusped hyperbolic 3-manifolds do not bound geometrically]{Many cusped hyperbolic 3-manifolds do not bound geometrically}

\begin{document}

\begin{abstract}
In this note, we show that there exist cusped hyperbolic $3$-manifolds that embed geodesically, but cannot bound geo\-met\-ri\-cal\-ly.
Thus, being a geometric boundary is a non-trivial property for such manifolds. Our result complements the work by Long and Reid on geometric boundaries of compact hyperbolic $4$-manifolds, and by Kolpakov, Reid and Slavich on embedding arithmetic hyperbolic manifolds. \\


\noindent
\textit{2010 AMS Classification: } 57R90, 57M50, 20F55, 37F20. \\
\end{abstract}

\maketitle

\section{Introduction}\label{intro}

In the sequel, all hyperbolic ma\-ni\-folds are assumed to be connected, orientable, complete, and of finite volume. We are particularly interested in \textit{cusped}, i.e. non-compact, such manifolds. 

A hyperbolic $n$-manifold $M$ \emph{bounds geometrically} if it is isometric to $\partial W$, for a hyperbolic $(n+1)$-manifold $W$ with totally geodesic boundary, c.f. \cite{LR}, and also \cite{KMT, KR, LR2, MZ, M, S, S2} for further progress in this topic.
A hyperbolic $n$-manifold $M$ is said to \emph{embed geodesically} if there exists a hyperbolic $(n+1)$-manifold $N$ that contains a totally geodesic hypersurface isometric to $M$. We remark that many arithmetic hyperbolic $3$-manifolds of simplest type embed geodesically by \cite{KRS}.

A geometrically bounding manifold embeds geodesically, but the converse is not necessarily true. Indeed, the Euler characteristic $\chi(M)$ of a geometrically bounding manifold $M$ must be even. This can be seen by taking a hyperbolic $(n+1)$-manifold $N$ with totally geodesic boundary $\partial N = M$, and doubling it along $M$ in order to obtain a hyperbolic manifold $DN$. By the excision property, we have for the Euler characteristic that $\chi(DN) = 2\, \chi(N) - \chi(M)$. If $n$ is odd, we have that $\chi(M) = 0$, while if $n$ is even, then $\chi(DN) = 0$, and $\chi(M)$ is thus even. The fact that an odd-dimensional cusped hyperbolic manifold has $\chi=0$ follows from Margulis' Lemma and the 1\textsuperscript{st} Bieberbach Theorem.

Thus, the thrice-punctured sphere cannot bound geometrically. On the other hand, this manifold is arithmetic and of even dimension, so by \cite{KRS} it embeds geodesically. This same discussion also applies when $n = 4,\, 6$,  since the respective minimal-volume arithmetic manifolds constructed in \cite{ERT, RT} have Euler characteristic $\chi = \pm 1$.
Note that such an argument becomes vacuous if $n$ is odd.

The aim of this note is to provide examples of hyperbolic 3-manifolds that embed geodesically, but fail to bound geo\-met\-ri\-cal\-ly, thereby explicitly showing that bounding is much more non-trivial to arrange for $n=3$ too. 

In particular, we show that several well-known cusped hyperbolic  $3$-manifolds cannot bound geometrically. Namely,  we prove the following theorems, the first of which should be contrasted with \cite{S2}, which shows that the figure-eight knot complement bounds geometrically.

\begin{theorem}\label{thm1}
The figure-eight knot sibling $3$-manifold embeds geodesically but does not bound geometrically.
\end{theorem}

The figure-eight knot complement and its ``sibling'' manifold are precisely the cusped hyperbolic $3$-manifolds of smallest volume \cite{CM}. Both of them are known to be arithmetic \cite{MaRe} with invariant trace-field $\mathbb{Q}(\sqrt{-3})$. Our methods also show:

\begin{theorem}\label{thm2}
A single-cusped hyperbolic $3$-manifold with invariant trace-field of odd degree does not bound geometrically.
\end{theorem}

There are many such examples of single-cusped hyperbolic $3$-manifolds, indeed even arising as knot complements in $S^3$.   We record the following corollary of Theorem \ref{thm2}.  This follows automatically from \cite{HS} which establishes that
if $K_m$ is the $m$-twist knot (see Figure~\ref{fig:knots}) the degree of the invariant trace-field is given by $\cro(K_m)-2$ where $\cro(K_m)$ is the crossing number of $K_m$ (here $m\neq -2,-1,0,1$). Note that the figure-eight knot is the $2$-twist knot according to Figure~\ref{fig:knots}. If we assume that $m\geq 2$, then $\cro(K_m)-2=m$.

\medskip

\begin{figure}[h]
\includegraphics[scale=0.45]{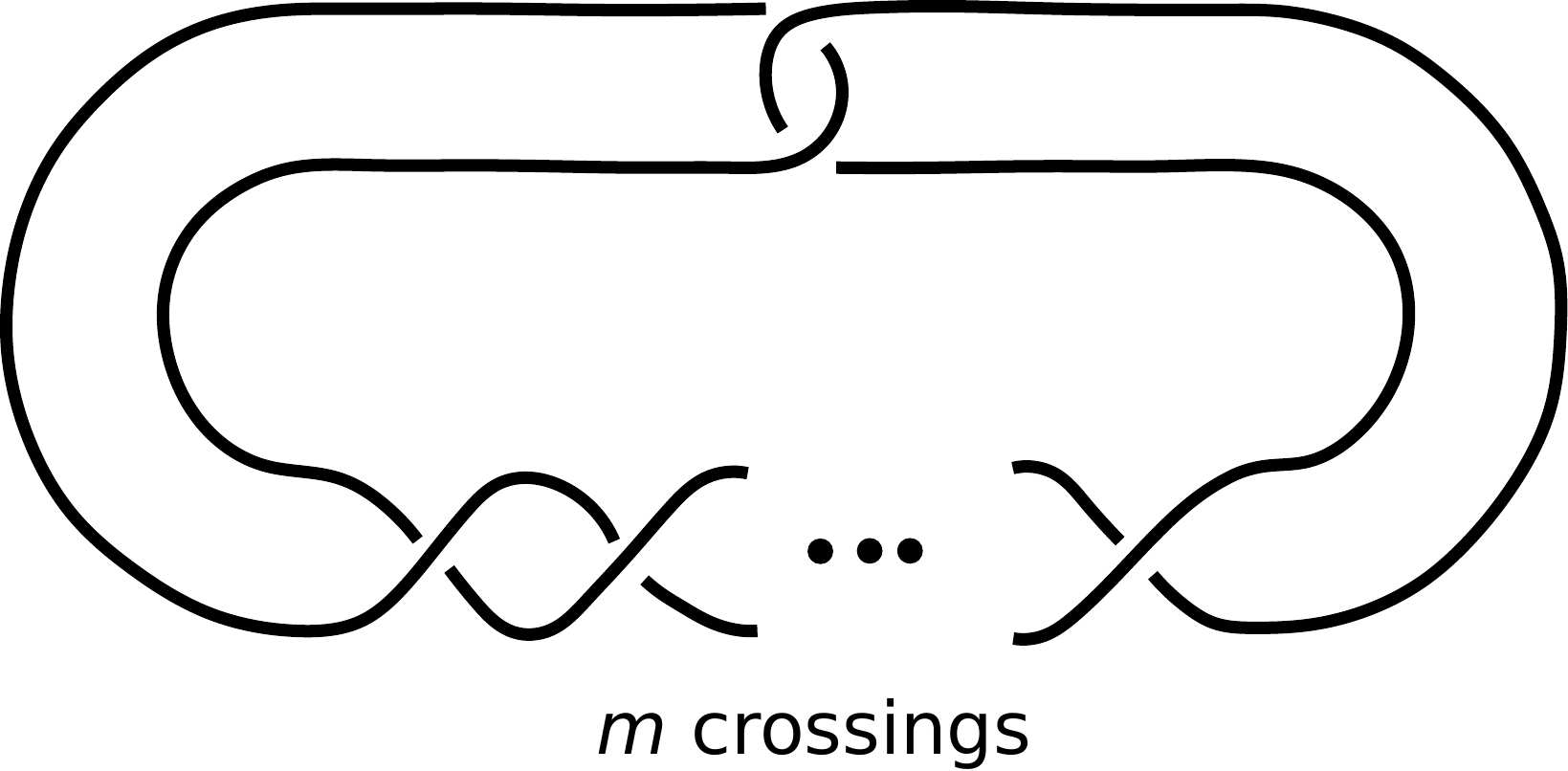}
\caption{The twist knot $K_m$ ($m>0$) in its alternating (and thus minimal) projection.}
\label{fig:knots}
\end{figure}

\begin{cor}
\label{twist}
Let $K_m$ be the $m$-twist knot with $m>1$ odd. Then $\mathbb{S}^3\smallsetminus K_m$ does not bound geometrically.\end{cor}

Note that  the figure-eight knot's sibling manifold does not satisfy Theorem~\ref{thm2}, since as noted above its invariant trace-field has degree two. In addition, at present we do not know {\em a single example} of a finite volume hyperbolic 3-manifold with odd degree invariant trace-field that even embeds geodesically.

The proofs of Theorems \ref{thm1} and \ref{thm2} essentially follow from a simple observation about the cusps of manifolds that bound geometrically, which applies in all dimensions. In the case of $n=3$, this implies that the cusp of a  single-cusped hyperbolic $3$-manifold that bounds geo\-met\-ri\-cal\-ly must be ``rectangular''(c.f. Proposition \ref{prop:cusp_bounding_3d}). We refer to Section \ref{sec:cusps} for the terminology.  In the case of $n=4$ this gives another proof that the minimal volume hyperbolic $4$-manifolds of \cite{RT} do not bound geometrically (see Section \ref{RT_notbound}).

\begin{rem}
Nimershiem \cite{N} proved that the cusp shapes of single-cusped hyperbolic $3$-manifolds form a dense subset in the moduli space of the $2$-torus.
Since the set of rectangular tori is nowhere dense, this lends credence to the claim that ``most single-cusped hyperbolic $3$-manifolds do not bound geometrically''.
Similar considerations were known to hold in the compact setting. Indeed, in \cite{LR} it is shown that if a closed hyperbolic $3$-manifold bounds geometrically, then it has integral $\eta$-invariant.  On the other hand, Meyerhoff \cite{Me} showed that a reduction
modulo $\frac{1}{3}$ of the $\eta$-invariant of closed hyperbolic $3$-manifolds takes values in a dense subset of the circle.\end{rem}

As a possible measure of geometric ``complexity'' of embedding geodesically a hyperbolic $n$-manifold $M$ into a hyperbolic $(n+1)$-manifold $N$, or making $M$ bound a hyperbolic $(n+1)$-manifold $W$ geometrically, we introduce the following quantities:
$$\evol(M)=\min_N\mathrm{vol}(N),\quad\bvol(M)=\min_W\mathrm{vol}(W),$$
where ``$\mathrm{vol}$'', here and below, means hyperbolic volume. It is easy to see that if $M$ bounds, then $\evol(M)\leq 2\cdot \bvol(M)$. 

In \cite{S2}, Slavich proved that the figure-eight knot complement has $\bvol=\frac{4\pi^2}3$, i.e. the minimum possible by the Gau\ss-Bonnet theorem.
Concerning $\evol$, we adopt his technique to improve Theorem \ref{thm1} as follows:

\begin{theorem}\label{thm4}
The figure-eight knot complement and its sibling manifold have $\evol=\frac{4\pi^2}{3}$.
\end{theorem}

Note that if $M$ bounds geometrically, resp. embeds geodesically, then Miyamoto's work \cite{Miyamoto} implies that
$\bvol(M) \geq d_{n+1}\, \mathrm{vol}(M)$, resp.  $\evol(M) \geq 2\, d_{n+1}\, \mathrm{vol}(M)$,
where $d_n$ is the optimal horoball packing density in $\mathbb{H}^n$, c.f. \cite[Table~3]{Kellerhals}. The former inequality holds also if $M$ is disconnected, and the latter follows by cutting any $N$ in which $M$ embeds along the respective hyper-surface isometric to $M$, and considering the resulting manifold $N_{\slash\!\!\slash} M$ with boundary $M \sqcup (-M)$. The equality, in the case of bounding manifolds, can be attained only in dimensions $n = 2$ and $3$. 

The paper is organised as follows: in Section \ref{sec:cusps} we prove Theorems \ref{thm1} and \ref{thm2}, in Section \ref{sec:sibling} we prove Theorem \ref{thm4}.

\begin{small}
\section*{Acknowledgments}
{\noindent
The authors were supported by the Swiss National Science Foundation project no.~PP00P2-170560 (A.K. and S.R.) and N.S.F. grant DMS-$1812397$ (A.R.). They would like to thank Danny Ruberman (Brandeis University, USA) and Bruno Martelli (Universit\`{a} di Pisa, Italy) for fruitful discussions.
}
\end{small}

\section{Cusp sections of geometric boundaries} \label{sec:cusps}

In this section, we provide a simple obstruction for a non-compact manifold to bound geometrically, formulated in Propositions \ref{prop:cusp_bounding} and \ref{prop:cusp_bounding_3d}, which we then use in order to prove Theorems~\ref{thm1} and \ref{thm2}. 

\subsection{Cusps with boundary}\label{subsec:cusps}
In this section, we analyse the ends of a hyperbolic manifold with totally geodesic boundary.  We basically follow \cite[2.10.C]{GPS}.

\begin{defn}
An $(n+1)$-dimensional \emph{cusp with boundary} is a Riemannian warped product $C = F \times_f (0,+\infty)$, where $F$ is a compact connected flat $n$-manifold with totally geodesic boundary and $f(r)=e^{-r}$. 
\end{defn}

This means that $F\times (0,+\infty)$ is endowed with the Riemannian metric $e^{-2r}g+dr^2$, where $g$ is the flat metric on $F$. 

\begin{defn}
A \emph{section} of a cusp $C$ as above is a level set $F \times \{r_0\} \subset C$. 
\end{defn}

Note that all sections of $C$ are homothetic. 

\begin{defn}
The \emph{shape} of a cusp $C$ is the similarity class of a section.
\end{defn}

The following fact is well known:

\begin{prop}{\cite[2.10.D]{GPS}}
Let $W$ be a complete, finite-volume hyperbolic $(n+1)$-manifold with (possibly empty) totally geodesic boundary.  There is a compact subset of $W$ whose complement is isometric to a disjoint union of cusps with (possibly empty) boundary.
\end{prop}

Given $W$ as above, we call \emph{boundary cusps} the cusps of $W$ with non-empty boundary. Each cusp of the hyperbolic $n$-manifold $\partial W$ is a boundary component of a boundary cusp of $W$.  

\subsection{Cusps of bounding manifolds}

We now furnish an obstruction for a hyperbolic manifold to bound geometrically.

\begin{prop} \label{prop:cusp_bounding}
If a cusped hyperbolic $n$-manifold $M$ bounds geometrically, then the cusps of $M$ that do not admit a fixed-point-free orientation-reversing isometric involution are isometric in pairs.
\end{prop}

In particular, the number of such cusps has to be even, possibly zero. 
The proof will follow easily from a simple lemma about flat manifolds with totally geodesic boundary:

\begin{lemma}\label{lem:flat_boundary}
Let $F$ be a compact connected orientable flat $n$-manifold with non-empty totally geodesic boundary, such that $F$ is not isometric to a product with an interval. Then, $\partial F$ is connected and has a fixed-point-free orientation-reversing isometric involution.
\end{lemma}
\begin{proof}
The Riemannian universal cover of any constant-curvature manifold with totally geodesic boundary embeds isometrically into the model space as the intersection of some half-spaces with pairwise disjoint boundaries. In the flat case, the number of such half-spaces can only be $0$, $1$, or $2$. The first case is excluded because $\partial F\neq\emptyset$, while the second one is excluded because $F$ is compact.

Thus, the universal cover $\widetilde F$ of $F$ is isometric to a strip $\matR^{n-1}\times I\subset\matR^n$, where $I=[-a,a]$, and $F$ is isometric to a quotient of this strip by a discrete group of Euclidean isometries acting on it.
Every isometry of the strip $\widetilde F$ must preserve the $I$-fibration and the $0$-section $\matR^{n-1}\times\{0\}$.
This implies that there is an $I$-bundle $\pi\colon F\to B$ whose $0$-section $B_0\subset F$ is a totally geodesic hypersurface. Note that $\pi\big|_{\partial F}$ is a Riemannian double covering.

By hypothesis the bundle $\pi$ is non-trivial, so $B_0$ is a one-sided hypersurface inside $F$, and $B$ is non-orientable. Thus, $\pi\big|_{\partial F}$ is the orientation double cover of $B$, and $\partial F$ has a fixed-point-free orientation-reversing involution. 
\end{proof}

\begin{rem} \label{rem:klein}
It follows from the proof of Lemma \ref{lem:flat_boundary} that if $n=3$ and $\partial F$ is connected, then $F$ is diffeomorphic to $K\widetilde{\times} I$, which is the orientable manifold arising as a twisted $I$-bundle over the Klein bottle.
\end{rem}

We are ready to prove Proposition \ref{prop:cusp_bounding}.

\begin{proof}[Proof of Proposition \ref{prop:cusp_bounding}]
Let $M=\partial W$ for a hyperbolic $(n+1)$-manifold $W$ with totally geodesic boundary, and let $C'\subset M$ be a cusp of $M$ with section $F'$.
Then, $F'\subset \partial F$ is a boundary component of a section $F$ of a boundary cusp $C$ of $W$.
If $F'$ has no fixed-point-free orientation-reversing involution, by Lemma~\ref{lem:flat_boundary} we have $F\cong F'\times I$, and $F'$ is isometric to a section of another cusp of $M$.
\end{proof}

\subsection{Rectangular tori}
In this section, we give a more precise characterisation of cusp shapes of geometrically bounding manifolds in the case $n=3$, and then prove Theorems \ref{thm1} and \ref{thm2}. 

A cusp of a hyperbolic $3$-manifold has section a flat $2$-torus. Recall that a flat torus $T=\matR^2/\Gamma$ has a fixed-point-free orientation-reversing isometric involution if and only if a conjugate of the lattice $\Gamma$ is generated by two vectors that span a rectangle or a rhombus. We call such flat tori respectively \emph{rectangular} or \emph{rhombic}.

In the usual fundamental domain
$$\mathcal{D}=\{|z|\geq1,\ |\mathrm{Re}(z)|\leq1/2,\ \mathrm{Im}(z)>0\}\subset\matC$$
for the moduli space of tori (c.f. for instance \cite[\S 12.2]{FM} and \cite[\S 4.2]{MNPR}), the rectangular and rhombic ones correspond to the curves $\mathcal{D}\cap\{\mathrm{Re}(z)=0\}$ and $\mathcal{D}\cap\{|z|=1\ \mathrm{or}\ |\mathrm{Re}(z)|=1/2\}$, respectively. Thus, we call a cusp of a hyperbolic 3-manifold \emph{rectangular} or \emph{rhombic} depending on the shape of its section.

With these definitions in hand, we can now improve Proposition \ref{prop:cusp_bounding}:

\begin{prop} \label{prop:cusp_bounding_3d}
If a cusped hyperbolic $3$-manifold $M$ bounds geometrically, then the non-rectangular cusps of $M$ are isometric in pairs.  
\end{prop}

\begin{proof}
Suppose again $M=\partial W$, and let $C\cong F\times_f(0,+\infty)$ be a boundary cusp of $W$ with connected boundary. By
Remark \ref{rem:klein}, $F$ is diffeomorphic to $K\widetilde{\times} I$ and $\partial F$ is a flat torus. We now show that $\partial F$ is rectangular.
To that end,  we have
$$F\cong\left(\matR^2\times[-a,a]\right)/\Gamma,$$
where we can assume that the group $\Gamma<\Iso(\matR^2\times[-a,a])<\Iso(\matR^3)$ is generated by a parallel translation $T_x$ along $(2b,0,0)$, a translation $T_y$ along $(0,c,0)$, and a roto-translation $R_x\colon(x,y,z)\mapsto(x+b,-y,-z)$ (c.f. for instance \cite[Theorem 3.5.5, item 2]{W}). In particular, we have
$$\partial F\cong\left(\matR^2\times\{a\}\right)/{\langle T_x,T_y\rangle},$$
and the lattice $\langle T_x,T_y\rangle$ is generated by two vectors spanning a rectangle.
\end{proof}
We can now prove Theorems \ref{thm1} and \ref{thm2}.

\begin{proof}[Proof of Theorem~\ref{thm1}:]
Let $\mathbb H^3/\Gamma$ be the figure-eight's sibling manifold. Up to conjugation, $\Gamma$ is an index $12$ subgroup in $PSL_2(\mathcal{O}_3)$, c.f. \cite[\S 13.7.1(vi)]{MaRe}. Thus $\Gamma$ can be embedded in $SO(q, \mathbb{Z})$ for a quadratic form of signature $(3, 1)$, c.f. \cite[\S 3]{Chu} and \cite{EGM, JM}. 
The argument given in \cite[\S 9.1]{KRS} now applies to show that the sibling manifold embeds geodesically. 

However, from \cite{Wie}, one sees that the modulus of the cusp is $\omega = \frac{-1+\sqrt{-3}}{2}$, and hence is not rectangular.
Thus, by Proposition~\ref{prop:cusp_bounding_3d}, the figure-eight knot's sibling manifold does not bound geometrically. 
\end{proof}

\begin{rem}
\label{sisiter_snappy} The cusp shape of the figure-eight knot sibling can also be computed using SnapPy \cite{SnapPy}.
By setting $\texttt{M = Manifold('m003')}$, where $\texttt{'m003'}$ is the entry for the figure-eight knot's sibling manifold in the Callahan-Hildebrand-Weeks census \cite{CHW}, and issuing the command $\texttt{M.cusp\_info(0).modulus}$, one sees that the cusp section of the sibling manifold is not rectangular (a numerical estimate suffices). \end{rem}

\begin{proof}[Proof of Theorem~\ref{thm2}:]
Let $T$ be the cusp section of a single-cusped hyperbolic $3$-manifold $M = \mathbb{H}^3/\Gamma$ whose invariant trace field has odd degree. We shall show that the odd degree assumption precludes $T$ from being rectangular. 
Thus, assume to the contrary that $T$ is rectangular. Let $K$ be the trace-field of $\Gamma$, and $k$ its invariant trace field, i.e. the trace field of the group $\Gamma^{(2)} = \langle g^2 \,\, | \,\, g \in \Gamma \rangle$.

By \cite[Theorem 4.2.3]{MaRe} we may assume that, up to conjugation, $\Gamma \subset PSL_2(K)$, and moreover $\pi_1(T) = \langle a, b \rangle$, with 
\begin{equation*}
a = \left( \begin{array}{cc}
1& 1\\
0& 1
\end{array}  \right),\,\,\, b = \left( \begin{array}{cc}
1& s\\
0& 1
\end{array}\right)
\end{equation*}
while there exists an element $x\in \Gamma$ such that
\begin{equation*}
x = \left( \begin{array}{cc}
1& 0\\
t& 1
\end{array}\right).
\end{equation*}
\noindent The complex number $s$ above is the modulus parameter for the torus $T$.

Since $a^2, b^2, x^2 \in \Gamma^{(2)}$, it follows that $\mathrm{tr}(a^2x^2)\in k$, hence $t \in k$. Furthermore, $\mathrm{tr}(b^2x^2) = 2+4st \in k$. Thus the cusp parameter $s$ of $M$ belongs to $k$, which implies that $s$ has odd degree over $\mathbb{Q}$. Hence
$F = \mathbb{Q}(s) \subset  k$ is a sub-field of $k$, having odd degree. 
By Proposition~\ref{prop:cusp_bounding_3d}, $T$ is a rectangular torus, and thus its modulus belongs to the imaginary axis $i \mathbb{R}$.  
However,  if $s = i r$, for some $r \in \mathbb{R}$, then $F$ is preserved by complex conjugation, implying that the real sub-field $F\cap \mathbb{R}$ has degree $2$ in $F$, 
contradicting the fact that the degree of $F$ is odd. This completes the proof.
\end{proof}

\subsection{Minimal volume hyperbolic 4-manifolds}\label{RT_notbound}
The Ratcliffe-Tschantz census \cite{RT} contains most of the known cusped hyperbolic $4$-manifolds of minimal volume. All of these manifolds are arithmetic. In particular, by \cite{KRS} they all embed geodesically.

By \cite[Table 2]{RT}, each of the $22$ Ratcliffe-Tschantz orientable $4$-manifolds has an odd number of cusps with section diffeomorphic to the so-called \emph{Hantzsche-Wendt manifold} (denoted by F in \cite{RT} and by $\mathcal G_6$ in \cite{W}). This flat $3$-manifold has no fixed-point-free orientation-reversing self-homeomorphism (c.f. \cite[Theorem 3.5.9]{W} and also \cite{Z}). Thus, by Proposition \ref{prop:cusp_bounding}, none of the manifolds from \cite[Table 2]{RT} bounds geometrically.\footnote{One can also arrive at this conclusion by the Euler characteristic argument mentioned in Section \ref{intro}.}

\section{Embedding the figure-eight's sibling} \label{sec:sibling}

Although Theorem \ref{thm1} shows that the figure eight knot sibling embeds in a hyperbolic $4$-manifold, it gives little control on the topology of the latter. The purpose of this section is to prove Theorem~\ref{thm4} using an approach due to Slavich \cite{S2}, which will afford additional control. We start with some necessary definitions.

\begin{defn}
A $4$-dimensional \emph{triangulation} $\mathcal{T}$ is a pair $(\{\Delta_i\}^{2k}_{i=1}, \{g_j\}^{5k}_{j=1})$, where $k$ is a positive integer, the $\Delta_i$'s are copies of the standard $4$-dimensional simplex, and the $g_j$'s are simplicial pairings between all the $10k$ facets of the $\Delta_i$'s. 
\end{defn}

\begin{defn}
A triangulation $\mathcal{T}$ is \emph{orientable} if it is possible to choose an orientation for each $\Delta_i$ so that all the $g_j$'s are orientation-reversing (c.f. also \cite[Definition 4.2]{KS}).
\end{defn}

\begin{defn}
A $4$-dimensional triangulation $\mathcal{T}$ is \emph{$6$-valent} if all cycles of $2$-faces in $\mathcal{T}$ have length exactly $6$.
\end{defn}

With each cycle $c$ of $2$-faces in $\mathcal{T}$ there is a naturally associated return map $r_c$ from a $2$-simplex to itself. In order to obtain it, one has to follow the simplicial pairings from one 4-simplex to the next one, until the cycle closes up. 

Our proof will make essential use of the fact stated below.

\begin{prop}[Proposition 3.9 in \cite{S2}]\label{prop:embedding}
Let $M$ be a hyperbolic $3$-manifold obtained by glueing the sides of some copies of the regular ideal hyperbolic tetrahedron via isometries.
If this glueing can be realised as the link of a vertex in a $6$-valent orientable $4$-dimensional triangulation $\mathcal{T}$ with trivial return maps, then $M$ embeds geodesically. Moreover,
$$\evol(M)\leq\frac{4\pi^2}{3}\cdot\frac{2k}3,$$
where $2k$ is the number of $4$-simplices in $\calT$ .
\end{prop}

\begin{proof}[Sketch of proof]
By replacing each $4$-simplex of $\mathcal{T}$ with an ideal hyperbolic rectified $5$-cell, one gets a hyperbolic $4$-manifold $W$ with totally geodesic boundary $\partial W$ tessellated by regular ideal tetrahedra. The link of each vertex of $\mathcal T$ gives the tessellation into tetrahedra of a boundary component of $W$. The manifold $M$ embeds geodesically in the double of $W$. Finally, the volume of the ideal rectified 4-simplex is $\frac{2\pi^2}{9}$ \cite{KS}. 
\end{proof}

We are ready to prove Theorem~\ref{thm4}.

\begin{proof}[Proof of Theorem~\ref{thm4}:]
We will adopt the usual ideal triangulations of the figure-eight knot complement and its sibling manifold by regular ideal hyperbolic tetrahedra. Each of them consists of two such tetrahedra, $A$ and $B$, with the following glueing maps between their $2$-faces.

For the figure-eight knot complement, depicted in Figure~\ref{fig:triangulations}--(i), we set:
\begin{equation}\label{eq:glueing-8-1}
\begin{array}{ccc}
A& \hspace*{0.25in}& B\\
(1,2,3)&  \leftrightarrow& (3,2,1)\\
(1,2,4)&  \leftrightarrow& (1,4,2)\\
(1,3,4)&  \leftrightarrow& (3,4,2)\\
(2,3,4)& \leftrightarrow& (4,1,3),
\end{array}
\end{equation}
while for the figure-eight sibling manifold, depicted in Figure~\ref{fig:triangulations}--(ii), we set:
\begin{equation}\label{eq:glueing-8-2}
\begin{array}{ccc}
A& \hspace*{0.25in}& B\\
(1,2,3)&  \leftrightarrow& (4,1,2)\\
(1,2,4)&  \leftrightarrow& (3,4,1)\\
(1,3,4)&  \leftrightarrow& (1,3,2)\\
(2,3,4)& \leftrightarrow& (2,4,3).
\end{array}
\end{equation}

\begin{figure}[h]
  \centering
  \begin{tabular}{@{}c@{}}
 \small (i)    
    \includegraphics[scale = .27]{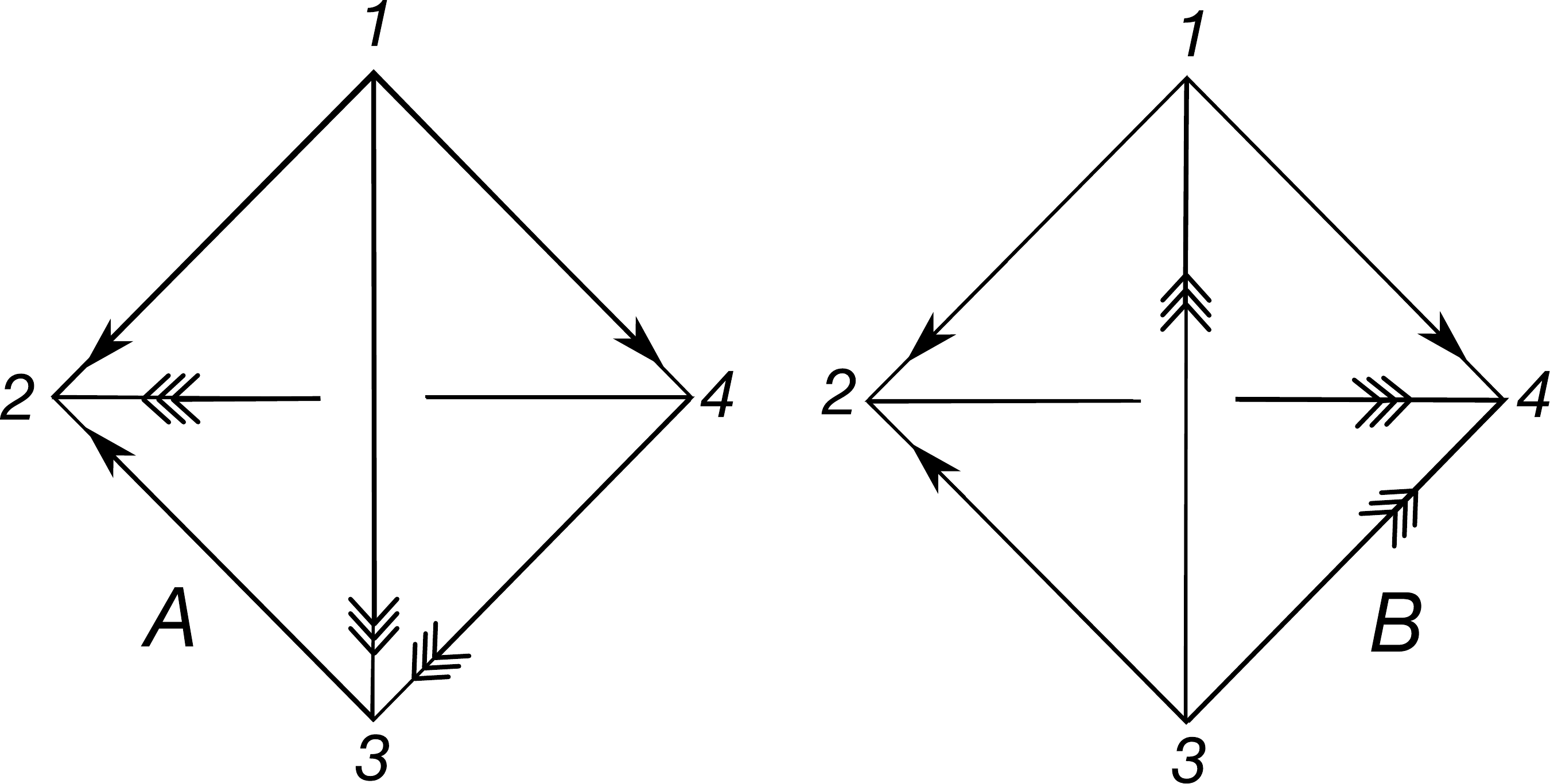} 
  \end{tabular}

  \begin{tabular}{@{}c@{}}
  \small (ii)
    \includegraphics[scale = .27]{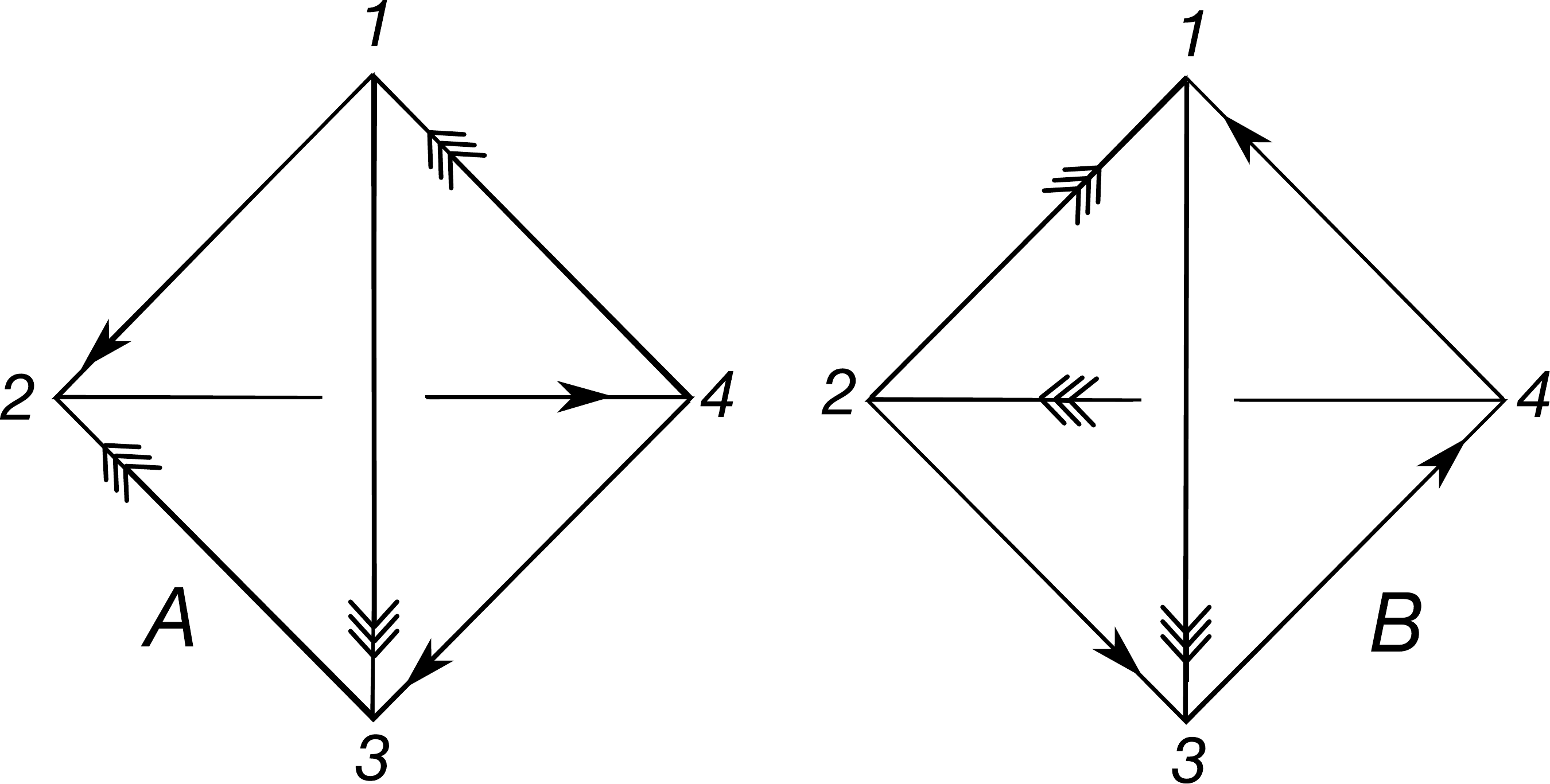} 
  \end{tabular}

  \caption{\footnotesize Ideal triangulations: (i) figure-eight knot complement; (ii) figure-eight sibling manifold.}\label{fig:triangulations}
\end{figure}

Let $Y$ be the cone over the 3-dimensional triangulation in Figure~\ref{fig:triangulations}--(i), and $X,Z$ be two copies of the cone over the triangulation in Figure~\ref{fig:triangulations}--(ii).

This means that each of $X$, $Y$ and $Z$ consists of two $4$-simplices $A^\prime$ an $B^\prime$ whose facets are identified as follows:
\begin{equation}\label{eq:glueing-2}
Y:  \left\lbrace \begin{array}{ccc}
A^\prime& \hspace*{0.25in}& B^\prime\\
(1,2,3,5)&  \leftrightarrow& (3,2,1,5)\\
(1,2,4,5)&  \leftrightarrow& (1,4,2,5)\\
(1,3,4,5)&  \leftrightarrow& (3,4,2,5)\\
(2,3,4,5)& \leftrightarrow& (4,1,3,5),
\end{array} \right.
\end{equation}
and
\begin{equation}\label{eq:glueing-1}
X, Z:  \left\lbrace \begin{array}{ccc}
A^\prime& \hspace*{0.25in}& B^\prime\\
(1,2,3,5)&  \leftrightarrow& (4,1,2,5)\\
(1,2,4,5)&  \leftrightarrow& (3,4,1,5)\\
(1,3,4,5)&  \leftrightarrow& (1,3,2,5)\\
(2,3,4,5)& \leftrightarrow& (2,4,3,5).
\end{array} \right.
\end{equation}

Observe that each of $X$, $Y$ and $Z$ has two remaining facets $A$ and $B$ with vertices $\{1,2,3,4 \}$ unidentified.
We shall build a 4-dimensional triangulation $\mathcal T$ by pairing these free facets of $X$, $Y$ and $Z$ as depicted in Figure~\ref{fig:diagram}.
\begin{figure}[h]
\includegraphics[scale=0.3]{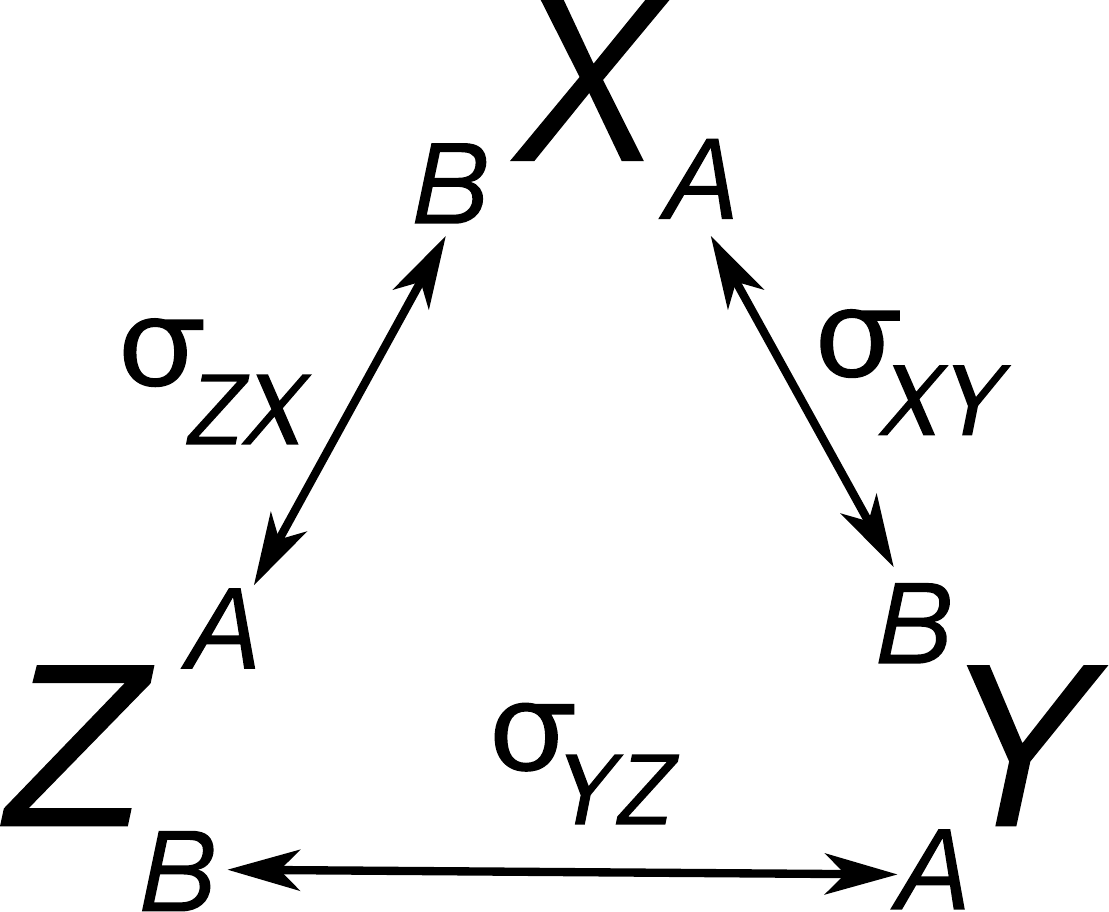}
\caption{\footnotesize Identifying the free facets of $X$, $Y$, and $Z$.}
\label{fig:diagram}
\end{figure}
The map $\sigma_{XY}$ will be used in order to identify facet $A$ of $X$ to $B$ of $Y$, and an analogous notation $\sigma_{YZ}$ and $\sigma_{ZX}$ is adopted for the remaining maps. We set:
\begin{equation}\label{eq:glueing-3}
\begin{array}{c}
\sigma_{XY}: (1,2,3,4) \rightarrow (3,1,4,2)\\
\sigma_{YZ}: (1,2,3,4) \rightarrow (3,4,2,1)\\
\sigma_{ZX}: (1,2,3,4) \rightarrow (2,4,1,3).
\end{array}
\end{equation}

Now we check that the resulting $4$-dimensional triangulation $\mathcal T$ satisfies the conditions of Proposition~\ref{prop:embedding}. First of all, $\calT$ is orientable because \eqref{eq:glueing-2}, \eqref{eq:glueing-1} are cones over orientable triangulations, and the pairing maps in \eqref{eq:glueing-3} are identified with odd permutations in the symmetric group $\mathfrak S_4$. Second, the condition on the cycles of $2$-faces should be satisfied.

By using the glueing equations \eqref{eq:glueing-2}, \eqref{eq:glueing-1} and \eqref{eq:glueing-3}, together with the diagram in Figure~\ref{fig:diagram}, we can compute the cycles of 2-faces with no vertex labelled 5:
\begin{eqnarray*}
X_A: (1,2,3) \rightarrow Y_B: (3,1,4) \rightarrow Y_A: (4,3,2) \rightarrow Z_B: (1,2,4) \rightarrow \\
\rightarrow Z_A: (2,3,1) \rightarrow X_B: (4,1,2) \rightarrow X_A: (1,2,3),
\end{eqnarray*}
\begin{eqnarray*}
X_A: (1,2,4) \rightarrow Y_B: (3,1,2) \rightarrow Y_A: (1,3,2) \rightarrow Z_B: (3,2,4) \rightarrow \\
\rightarrow Z_A: (4,2,3) \rightarrow X_B: (3,4,1) \rightarrow X_A: (1,2,4),
\end{eqnarray*}
\begin{eqnarray*}
X_A: (1,3,4) \rightarrow Y_B: (3,4,2) \rightarrow Y_A: (1,3,4) \rightarrow Z_B: (3,2,1) \rightarrow \\
\rightarrow Z_A: (3,4,1) \rightarrow X_B: (1,3,2) \rightarrow X_A: (1,3,4),
\end{eqnarray*}
\begin{eqnarray*}
X_A: (2,3,4) \rightarrow Y_B: (1,4,2) \rightarrow Y_A: (1,2,4) \rightarrow Z_B: (3,4,1) \rightarrow \\
\rightarrow Z_A: (1,2,4) \rightarrow X_B: (2,4,3) \rightarrow X_A: (2,3,4).
\end{eqnarray*}
All such cycles have length $6$ and trivial return maps.
The same conclusion holds for the cycles of $2$-faces containing vertex $5$, since they correspond to the glueing of edges of simplices $A$ and $B$ in the manifold triangulations from Figure~\ref{fig:triangulations}.

By using Regina \cite{Regina} we can conclude that $\mathcal{T}$ has $4$ vertices.\footnote{A word of caution to the reader: Regina does not recognise $\mathcal{T}$ as a valid triangulation, since it does not allow reverse identifications of edges. However, $\mathcal{T}$ does not have to satisfy this condition. The links of vertices are valid triangulations for Regina, as it should be.} Two of their links are isomorphic to the sibling manifold triangulation, one to the figure-eight triangulation, and the remaining fourth link is isomorphic to the triangulation of the manifold $O = \mathtt{otet24\_00260}$ from the census of tetrahedral manifolds \cite{census}. Thus, as described in the proof of Proposition \ref{prop:embedding}, we have
$$\partial W\cong K\sqcup L\sqcup L\sqcup O,$$
where $K$ is the figure-eight knot complement, and $L$ is its sibling manifold.

The figure-eight knot complement $K$ is the orientation double-cover of the non-orientable Gieseking manifold, while $O$ is the orientation double-cover of the non-orientable manifold $\mathtt{ntet12\_00019}$ from \cite{census} (as one can verify by SnapPy).
Thus, we can quotient the $O$ and $K$ boundary components of $W$, in order to obtain a hyperbolic manifold $W'$ with two boundary components, each isometric to $L$. By identifying the $L$ boundary components of $W'$, we obtain a hyperbolic $4$-manifold $N$ of volume $\frac{4 \pi^2}{3}$, in which the sibling manifold $L$ embeds geodesically.

Similarly, Slavich produced a hyperbolic 4-manifold $W''$ with totally geodesic boundary
$$\partial W''\cong K\sqcup K\sqcup K\sqcup O'$$
with $\mathrm{vol}(W'')=\frac{4\pi^2}{3}$, where $O'$ is another tetrahedral 3-manifold with an orientation-reversing fixed-point-free involution \cite[Remark 4.4]{S2}. To conclude the proof for the figure-eight knot complement, we glue together two $K$-components of $\partial W''$ via an isometry, and quotient the remaining boundary components as before.

\end{proof}

\end{document}